\documentclass[reqno, 10.5pt]{amsart}
\usepackage{amsmath}
\usepackage{amscd}
\usepackage{amssymb}
\usepackage{comment}
\usepackage{eqnarray}

\newcommand{\na}{\nabla}

\newcommand{\vol}{\text{Vol}}

\newtheorem{theorem}{Theorem}[section]

\newtheorem{lemma}{Lemma}[section]

\newtheorem{proposition}{Proposition}[section]

\numberwithin{equation}{section}
\theoremstyle{remark}
\newtheorem{remark}{Remark}[section]

\theoremstyle{definition}

\begin{document}
\title[Ricci Expanders with finite asymptotic scalar curvature ratio]
{Complete gradient expanding Ricci solitons with finite asymptotic scalar curvature ratio}
\author{HUAI-DONG CAO$^{\dagger}$, Tianbo Liu and Junming Xie}
\address{Department of Mathematics,  Lehigh University,
Bethlehem, PA 18015, USA}
\email{huc2@lehigh.edu; til615@lehigh.edu; jux216@lehigh.edu}

\thanks{$^{\dagger}$Research partially supported by a Simons Foundation Collaboration Grant (\#586694 HC)}

\begin{abstract}
	Let $(M^n, g, f)$, $n\ge 5$, be a complete gradient expanding Ricci soliton with nonnegative Ricci curvature $Rc\geq 0$. In this paper, we show that if the asymptotic scalar curvature ratio of $(M^n, g, f)$ is finite (i.e., $ \limsup_{r\to \infty} R  r^2< \infty $), then the Riemann curvature tensor must have at least sub-quadratic decay, namely, $\limsup_{r\to \infty} |Rm| \ \! r^{\alpha}< \infty$ for any $0<\alpha<2$. 
\end{abstract}

\maketitle
\date{}

\section{Introduction}

This is a sequel to the earlier paper \cite{CaoLiu21} by the first and the second authors in which curvature estimates were obtained for $4$-dimensional complete gradient expanding Ricci solitons.  
By scaling the metric $g$ if necessary, we shall assume throughout the paper that a gradient expanding Ricci soliton $(M^n, g, f)$ satisfies the  equation 
\begin{equation} \label{expandingeq}
        Rc+\nabla^2 f=-\frac 1 2  g, 
\end{equation}
where $Rc$ and $\nabla^2 f$ denote the Ricci tensor of $g$ and the Hessian of the potential function $f\in C^{\infty}(M)$, respectively. 

For any $4$-dimensional complete gradient expanding Ricci soliton $(M^4, g, f)$ with nonnegative Ricci curvature  $Rc\ge 0$, it was shown in \cite{CaoLiu21} that there exists a constant $C>0$  such that, for any 
$0\leq a<1$, the following curvature estimate hold on $M^4$, 
\begin{equation*} 
|Rm|  \le \frac {C} {1-a} R^a.
\end{equation*} 
Moreover,  if the scalar curvature $R$ has at most polynomial decay, then 
\begin{equation} 
{|Rm|} \le C R  \quad on \ M^4.
\end{equation}
On the other hand, if the asymptotic scalar curvature ratio of $(M^4, g)$ is finite, i.e.,
	\begin{equation*} 
	\limsup_{r\to \infty} R  r^2< \infty, 
	\end{equation*} 
 then $(M^4, g)$ has finite  asymptotic curvature ratio
\begin{equation} 
		A:= \limsup_{r\to \infty} |Rm| r^2< \infty. 
\end{equation}
As an application, it follows from the above result and the work of Chen-Deruelle \cite{ChenDer} that any 4-dimensional complete noncompact non-flat gradient expanding Ricci soliton with nonnegative Ricci curvature and finite asymptotic scalar curvature ratio must have a $C^{1,\alpha}$ asymptotic cone structure at infinity, for any $\alpha \in (0, 1)$.

We remark that recent progress on curvature estimates for 4-dimensional gradient Ricci solitons has been led by the work of Munteanu-Wang \cite{MW15},  in which they proved that any complete gradient shrinking soliton with bounded scalar curvature  $R$ must have bounded Riemann curvature tensor $Rm$. More significantly, they showed that the Riemann curvature tensor is controlled by the scalar curvature by $|Rm|\le C R$ so that if the scalar  curvature $R$ decays at infinity so does the curvature tensor $Rm$. Moreover, by exploring the differential equation $\Delta_f R=R-2|Rc|^2$ satisfied by shrinking solitons and combining with the scalar curvature lower bound of Chow-Lu-Yang \cite{CLY}, they showed that the scalar  curvature $R$ in fact must decay quadratically if $R$ goes to zero at infinity. It then follows that the curvature tensor $Rm$ must decay quadratically, hence the 4D shrinking soliton is asymptotically conical. Their curvature estimate, together with the uniqueness result of Kotschwar-Wang \cite{KW15}, has played a crucial role in the recent advance of classifying 4-dimensional complete gradient Ricci solitons, as well as in the classification of complex 2-dimensional complete gradient K\"ahler-Ricci solitons with scalar curvature going to zero at infinity by Conlon-Deruelle-Sun \cite{CDS19}. See  \cite {Cao et al2} for an extension, and also \cite{CaoCui, Chan1} and \cite{Cao2021} for similar curvature estimates in the steady soliton case.

\medskip
In  \cite{MW17}, via the Moser iteration and a tour de force of integral estimates, Munteanu and Wang also obtained the curvature estimate for higher dimensional gradient shrinking Ricci solitons. Precisely, they showed that if the Ricci curvature of an $n$-dimensional ($n\ge 5$) complete gradient {\em shrinking Ricci soliton} goes to zero at infinity, then its Riemann curvature tensor $Rm$ must also go to zero at infinity. Furthermore, based on $|Rm| \to 0$ at infinity and the fact that the curvature tensor of Ricci shrinkers satisfy the differential inequality $\Delta_f |Rm| \geq |Rm| -c |Rm|^2$, they were able to show that $Rm$ has to decay quadratically at infinity by using the maximum principle argument. 

In this paper, inspired by the work of Munteanu-Wang \cite{MW17}, we investigate curvature estimates for higher dimensional gradient expanding Ricci solitons with nonnegative Ricci curvature and finite 
asymptotic scalar curvature ratio. Our main result is the following

\begin{theorem} \label{maintheorem}
	Let $(M^n, g, f)$,  $n\ge 5$, be an $n$-dimensional complete gradient expanding Ricci soliton with nonnegative Ricci curvature $Rc\geq 0$ and finite asymptotic scalar curvature ratio
	\begin{equation} \label{fas}
		\limsup_{r\to \infty}  R \ \!r^2< \infty,
	\end{equation}
	where $r=r(x)$ is the distance function to a fixed base point $x_0\in M$. 
	Then $(M^n, g, f)$ has finite $\alpha$-asymptotic curvature ratio for any $0<\alpha<2$, 
	\begin{equation} \label{maindecay}
		A_{\alpha} := \limsup_{r\to \infty} |Rm| \ \! r^{\alpha}< \infty.
	\end{equation}
Furthermore, there exist constant $C>0$ depending on $n$ and the geometry of $(M^n, g, f)$, sequences $\{r_j\} \to \infty$ and $\{\alpha_j\} \to 2$ such that 
\begin{equation*} 
	|Rm|(x) \leq  C (r(x)+1)^{-\alpha_j}
\end{equation*}
for any $x \in M\setminus B(x_0, r_j+1)$. 
\end{theorem}

We point out that, compared to the shrinking case, there are several essential differences in the expanding case. 
First of all, certain integrals in the key integral estimate, which were good terms in the shrinking case, turned into potential trouble terms (see Remark \ref{remark3.2}) and they prevented us from obtaining the pointwise $Rm$ decay by merely assuming the Ricci curvature goes to zero at infinity. Secondly, in the expanding case, the assumption of $Rc\ge 0$ is essential  to ensure a uniform lower bound for the Sobolev constant of unit geodesic balls $B_x(1)$ for all $x\in M$ (see Lemma \ref{sobolev} and Lemma \ref{avr}), or a uniform non-collapsing estimate for $B_x(1)$ (see Lemma \ref{non-collapsing}), that is crucial for the Moser iteration to work. Finally, the corresponding differential inequality for $|Rm|$  in the expanding case becomes
\[\Delta_f |Rm| \geq -|Rm| -c |Rm|^2,\]
from which the maximum principle argument does not seem to work for getting the quadratic decay as in the shrinking case, or any improved decay for $Rm$, even knowing $Rm$ goes to zero at infinity.
Nevertheless, by adapting the integral estimates in \cite{MW17} and using the Moser iteration, we are able to obtain the sub-quadratic decay 
(\ref{maindecay}) for $|Rm|$ under the assumption of finite asymptotic scalar curvature ratio (\ref{fas}).

\begin{remark}
The same proof can be used to show that if
the rate of decay for the scalar curvature $R$ is in the order of $\alpha$, with $0<\alpha\leq 2$, then $Rm$ would have sub-$\alpha$ decay (see Theorem \ref{subalphadecaythm}).
\end{remark}

\begin{remark}
Unlike the works of Munteanu and Wang \cite{MW15, MW17} for shrinking Ricci solitons,  it seems that the best one can hope to prove in the expanding case is for $Rm$ to have the same decay rate as assumed for the Ricci curvature (or the scalar curvature). It remains an interesting question if one can improve the arbitrary sub-quadratic decay for $Rm$ in Theorem \ref{maintheorem} to the quadratic decay.  
\end{remark}

\medskip 
\noindent {\bf Acknowledgements.} We would like to thank Ovidiu Munteanu and Jiaping Wang for their interests in this work and their helpful comments and suggestions.  We are also grateful to the referee for the careful reading of our paper and for providing valuable suggestions which led to a simpler version of Lemma 3.1 and a more streamlined proof of Lemma 3.2 and Lemma 3.3 than in the previous version.

\section{Preliminaries}

In this section, for the reader's convenience, we fix the notations and collect several known results about gradient expanding Ricci solitons that we shall need later. Throughout the paper, we denote by $$Rm=\{R_{ijkl}\}, \quad Rc=\{R_{ik}\},\quad R $$ the Riemann curvature tensor, the Ricci tensor, and the scalar curvature of the metric $g=g_{ij}dx^idx^j$ in local coordinates $(x^1, \cdots, x^n)$, respectively.

\begin{lemma} {\bf (Hamilton \cite{Ha95F})} Let $(M^n, g, f)$
	be a complete gradient expanding Ricci soliton satisfying Eq. (1.1).
	Then
	$$ R+\Delta f =-\frac n 2,$$
	$$\nabla_iR=2R_{ij}\nabla_jf, $$
	$$R+|\nabla f|^2=-f +C_0 $$ for some constant $C_0$.
\end{lemma}
Moreover, replacing $f$ by $f-C_0$, we can normalize the potential function $f$ so that 
\[R+|\nabla f|^2=-f. \]
In the rest of the paper, we shall always assume this normalization.

Furthermore, by setting
$$F=-f+\frac n 2, \eqno (2.1)$$ the expanding soliton equation (1.1) becomes 
$$ \na^2 F =Rc+\frac 1 2 g. \eqno (2.2)$$
From (2.2), Lemma 2.1 and the normalization of $f$, we have 
$$ \nabla R=-2Rc \ \!(\nabla F, \cdot), \qquad |\na F|^2=F-R-\frac n 2, \eqno(2.3)$$
$$ \Delta F= R+ \frac n 2 \qquad  \mbox{and} \qquad \Delta_f F =F \quad ({\mbox {i.e.}},  \ \Delta_f f=f-\frac n 2), \eqno (2.4)$$ 
where $\Delta_f =:\Delta -\nabla f\cdot \nabla$ is the weighted Laplace operator.

Next, we have the following well-known fact about the asymptotic behavior of the potential function of a complete non-compact gradient expanding soliton with nonnegative Ricci curvature (see, e.g., Lemma 5.5 in \cite{Cao et al} or Lemma 2.2 in \cite{ChenDer}).

\begin{lemma} \label{potencialfunction} Let $(M^n, g , f)$ be a complete noncompact gradient expanding Ricci soliton satisfying Eq. (1.1) and with nonnegative Ricci curvature $Rc \ge 0$. Then there exist some constants $c_1 >0$ such that, outside some compact subset of $M^n$, the function $F=-f+n/2$ satisfies the estimates
	\[\frac 1 4(r(x)-c_1)^2\le F(x)\le \frac 1 4 (r(x)+2\sqrt{F(x_0)})^2,   \eqno(2.5)\]
	where $r(x)$ is the distance function from a base point in $M^n$. In particular, $F$ is a strictly convex exhaustion function achieving its minimum at its unique interior point $x_0$, which we shall take as the base point, and the underlying manifold $M^n$ is diffeomorphic to ${\mathbb R}^n$.
	
\end{lemma}

Another useful fact is the boundedness of the scalar curvature of a gradient expanding soliton with nonnegative Ricci curvature (see, e.g.,  Ma-Chen \cite{MC}).

\begin{lemma} 
	Let $(M^n, g, f)$ be a complete noncompact gradient expanding Ricci soliton with nonnegative Ricci curvature $Rc \ge 0$. Then its scalar curvature $R$ is bounded from above, i.e., $R\le R_0$ for some positive constant $R_0$. 
	Moreover, $R>0$ everywhere unless  $(M^n, g, f)$ is the Gaussian expanding soliton.
\end{lemma}

Note that, under the assumption of $Rc\ge 0$ (or even $Rc\ge (\epsilon-\frac 12)g$ for some constant $\epsilon>0$), the potential function $F(x)$ defined by (2.1) grows quadratically hence is proportional to $r^2(x)$, the square of the distance function, from above and below at large distance. In the rest of the paper, we denote by
\begin{eqnarray*}
	D(r)&=& \{x\in M : F(x) \leq r \}.
\end{eqnarray*}
By the Bishop volume comparison, we know that the volume $V(r)$ of $D(r)$ satisfies
$$ V(r) \leq cr^{\frac{n}{2}}. \eqno (2.6) $$

We now collect several well-known differential identities on the curvatures $R, Rc$ and $Rm$ that we shall use later. 

\begin{lemma} \label{diffequation}
	Let $(M^n, g, f)$ be a complete gradient expanding Ricci soliton satisfying Eq. (1.1). Then, we have
	\begin{eqnarray*}
		\Delta_{f} R &=&-R-2|Rc|^2,\\
		\Delta_{f} R_{ik} &=&-R_{ik} -2R_{ijkl}R_{jl},\\
		\Delta_{f} {Rm} &=&  -Rm+ Rm\ast Rm,\\
		\na_lR_{ijkl} &=& \na_jR_{ik}-\na_i R_{jk}=-R_{ijkl}\na_lF, 
	\end{eqnarray*}
	where,  on the RHS of the third equation, $Rm\ast Rm$ denotes the sum of a finite number of terms involving quadratics in $Rm$.
\end{lemma}

Based on Lemma 2.4,  one can easily derive the following differential inequalities  (see also \cite{MW15, CaoCui} for the shrinking and steady ones):

\begin{lemma} \label{diffeqofnorm}
	Let $(M^n, g, f)$ be a complete gradient expanding Ricci soliton satisfying Eq. (1.1). Then
	\begin{eqnarray*}
		\Delta_{f} |Rc|^2 & \ge & 2|\na Rc|^2-2|Rc|^2-4|Rm| |Rc|^2, \\
		\Delta_{f}|Rm|^2  &\ge & 2|\na Rm|^2 - 2|Rm|^2-c|Rm|^3,\\
		\Delta_{f} |Rm| &\ge & -|Rm|-c|Rm|^2.
	\end{eqnarray*}
	Here $c>0$ is some universal constant depending only on the dimension $n$.
\end{lemma}

Also,  from \cite{CaoLiu21} we have the following differential inequalities on the covariant derivative $\na Rm$ of the curvature tensor (see \cite{MW15} for the shrinking case). 

\begin{lemma} Let $(M^n, g, f)$ be a complete gradient expanding Ricci soliton satisfying Eq. (1.1). Then
	\begin{eqnarray*}
		\Delta_{f} |\na Rm|^2 &\ge & 2|\na^2 Rm|^2  -3|\na Rm|^2-c|Rm| |\na Rm|^2 \quad \mbox{and} \\
		\Delta_{f} |\na Rm| &\ge & -\frac 3 2|\na Rm|-c|Rm||\na Rm|.
	\end{eqnarray*}
\end{lemma}

In \cite{CarNi},  Carrillo and Ni proved the following non-collapsing result for gradient expanding soliton with nonnegative Ricci curvature.
 
\begin{lemma}  {\bf (Carrillo-Ni \cite{CarNi})} \label{non-collapsing} Let $(M^n, g, f)$ be a complete gradient expanding Ricci soliton with nonnegative Ricci curvature. Then there exists a constant $\kappa>0$ such that if $|Rc| \leq 1$ on a unit geodesic ball $B(x_0, 1)$ centered at $x_0$, then 
	\[V (x_0, 1)\geq \kappa,\] where $V (x_0, 1)$ denotes the volume of $B(x_0, 1)$.  
\end{lemma}

\begin{remark} In \cite{CarNi}, the authors only stated the non-collapsing result for shrinking Ricci solitons (see Corollary 4.2 in \cite{CarNi}). But Lemma 2.7 holds similarly because of their logarithmic Sobolev inequality for expanding solitons with nonnegative Ricci curvature (see Theorem 5.2 in \cite{CarNi}). 
\end{remark}

Concerning the volume growth, Hamilton \cite {Ha05} obtained the following result  (see also Proposition 9.46 in \cite{CLN}). 

\begin{lemma}  {\bf (Hamilton \cite{Ha05})} \label{avr} Let $(M^n, g, f)$ be any $n$-dimensional complete noncompact gradient expanding Ricci soliton with nonnegative Ricci curvature. Then it must have positive {\it asymptotic volume ratio}. Namely, for any base point $x_0\in M^n$, 
$$ \nu_M:= \lim_{r\to \infty} \frac {V (x_0, r)}{r^n} >0, \eqno (2.7) $$
where $V(x_0, r)$ denotes the volume of the geodesic ball $B(x_0, r)$. 
 \end{lemma}

Finally, we shall need the following well-known result about Sobolev inequality on manifolds with nonnegative Ricci curvature and positive asymptotic volume ratio; 
 see Yau \cite{Yau82}, and the very recent work of Brendle \cite{Brendle} for a sharp version.  

\begin{lemma}  \label{sobolev}
 Let $(M^n, g)$ be an $n$-dimensional complete manifold with nonnegative Ricci curvature $Rc\geq 0$ and positive asymptotic volume ratio  $\nu_M>0$. Then there exists a constant $C_s>0$ such that, for any compact domain $\Omega \subset M$
and any positive smooth function $\varphi$ with compact support in $\Omega$, 
\[ C_s \left(\int_{\Omega} \varphi^{\frac n {n-1}}\right)^{\frac {n-1}n} \leq \int_{\Omega} |\na \varphi|.\]
\end{lemma} 

\medskip

\section{The Integral  Estimate}

In this section, we prove a crucial integral curvature estimate needed in the proof of Theorem 1.1. First of all, 
we note that the assumptions of $Rc \geq 0$ and the finite asymptotic scalar curvature ratio (1.4) imply that

\smallskip
\begin{enumerate}
	\item [(i)] $0\leq R \leq R_0. $

\smallskip
	\item [(ii)] $ \na_i\na_j F \geq \frac{1}{2}g_{ij}. $
\smallskip

	\item [(iii)] $ |Rc| \leq R \leq \frac{C}{F}, $ for some constant $C>0$.

\smallskip
\end{enumerate}

\noindent Next, following \cite{MW17}, we define the cut-off function $\phi$ with support in $D (r)$ by
\begin{equation} \label{cut-off}
	\phi \left( x\right) =\left\{ 
	\begin{array}{ccc}
		\frac{1}{r}\left( r-F\left( x\right) \right) & \text{if} & x\in D\left(
		r\right) \\ 
		0 & \text{if} & x\in M\backslash D\left( r\right)%
	\end{array}%
	\right.
\end{equation}
so that
\[ \na \phi=-\na F/r \quad \mbox{and} \quad \Delta \phi =-\Delta F/ r \quad   \mbox{on} \ D(r). \]

Also, for any large number $p>0$ to be chosen later, we let $q=2p$ and pick $r_0 > 0$ sufficiently large  so that
\begin{equation} \label{decayofRic}
	F \geq p^5 \quad \text{and} \quad	|Rc| \leq \frac{1}{p^5} \quad \text{on}\ M \backslash D(r_0).
\end{equation}
In the rest of the paper, we shall use the following conventions. 

\begin{enumerate}
	\item[$\bullet$] $C$: a positive constant that may depend on the geometry of $D(r_0)$.

	\item[$\bullet$] $c$:  a positive constant depending only on the dimension $n$ and $R_0$.

	\item[$\bullet$] $c(p)$: a positive constant depending on $p$, $c$ and $C$.

\end{enumerate}
In addition, those constants may change from line to line.

\smallskip
Now we are ready to state our key integral curvature estimate. 
\begin{proposition} \label{prop}
	Let $(M^n, g, f)$ be an $n$-dimensional complete gradient expanding Ricci soliton with nonnegative Ricci curvature $Rc\geq 0$ and finite 
	{\it asymptotic scalar curvature ratio}
	$$\limsup_{r\to \infty}  R \ \!r^2< \infty. $$
	Then, for any constant $a>0$, there exists a constant $c\ge 1$ such that if $p>a+R_0+\frac{n}{2}+c$
we have
	\begin{equation}
		[1-p^{-1}(a+R_0+\frac{n}{2}+c)]\int_{M}\left\vert Rm \right\vert ^{p}F^{a}\phi ^{q} \leq c(p),
	\end{equation}	
where $c(p)$ is in the order of $p^p$. 
\end{proposition}

\smallskip
\begin{remark} Proposition 3.1 actually holds for gradient expanding Ricci solitons with finite asymptotic Ricci curvature ratio, i.e.,  $\limsup_{r\to \infty}  |Rc| \ \!r^2< \infty $, without assuming $Rc \ge 0$. Indeed, as we shall see below in the proof of Proposition 3.1, the condition of $Rc \geq 0$ is basically used only to guarantee that geodesic balls have at most polynomial (Euclidean) volume growth. 
However, note that any complete Riemannian manifold with quadratic Ricci curvature decay from below has polynomial volume growth (see \cite{CheeGT})\footnote{See also Corollary 4.11 in the very recent work of Chan-Ma-Zhang \cite{CMZ2022}.}. Meanwhile, under the finite asymptotic Ricci curvature ratio assumption, all the relevant properties or differential inequalities concerning the potential function $F$ would still hold outside a compact set. Hence, the nonnegative Ricci assumption in Proposition 3.1 is not essential.   
\end{remark}

We shall divide the proof of Proposition \ref{prop} into several lemmas and adapt the arguments in \cite{MW17}. 

\begin{lemma} \label{lemma1}
	Let $(M^n, g, f)$ be an $n$-dimensional complete gradient expanding Ricci soliton with nonnegative Ricci curvature $Rc\geq 0$.
	Suppose $p>a+R_0+\frac{n}{2}+1$, then
	\begin{eqnarray*}
	 [1-p^{-1}(a+R_0+\frac{n}{2})]\int_{M}\left\vert Rm \right\vert ^{p}F^{a}\phi ^{q} 
        &\leq & 4 \int_{M} |\na Rc|^2|Rm|^{p-1}F^{a+1}\phi^q \\
 &&  + \ cp^2\int_M |\na Rm|^2|Rm|^{p-3} F^{a-1}\phi^{q}\\
&& + \ c(p).     
	\end{eqnarray*}
\end{lemma}

\begin{remark} \label{remark3.2} In Lemma 3.1, the first two terms on the right hand side of the inequality are different from the shrinking case in \cite{MW17}. Also note that Lemma 3.1 does not require the {\it finite asymptotic scalar curvature ratio} assumption.
\end{remark} 

\begin{proof}
	Since $\Delta F \leq R_0 +n/2$ by (2.4) and Lemma 2.3, by integration by parts, we have  
	\begin{eqnarray*}
		-(R_0+\frac{n}{2})\int_{M}\left\vert Rm \right\vert ^{p}F^{a}\phi ^{q} 
		&\leq & -\int_{M} (\Delta F) \left\vert Rm \right\vert ^{p}F^{a}\phi ^{q} \\
		&=&  \int_M \na F \cdot \na(|Rm|^p)F^a\phi^q \\
		&& +  a \int_M  |Rm|^p |\na F|^2 F^{a-1}\phi^q \\
		&& +  q \int_M  |Rm|^p F^{a}\phi^{q-1} \na F\cdot\na \phi \\ 
		&\leq & \int_M \na F \cdot \na(|Rm|^p)F^a\phi^q + a\int_M |Rm|^pF^a\phi^q,
	\end{eqnarray*}
	where in the last inequality we have used the fact $|\na F|^2 <F$ from (2.3) and $\na \phi=-\na F/r$. 
	It then follows from the second Bianchi identity, as in \cite{MW17},  that 
	\begin{eqnarray*}
		-(a+R_0+\frac{n}{2})\int_{M}\left\vert Rm \right\vert ^{p}F^{a}\phi ^{q} 
		&\leq& \int_M \na F \cdot \na(|Rm|^p)F^a\phi^q \\
		&=& p\int_M (\na_hF \cdot \na_hR_{ijkl}) R_{ijkl}|Rm|^{p-2}F^a\phi^q \\
		&=& 2p\int_M (\na_h F \cdot \na_lR_{ijkh}) R_{ijkl}|Rm|^{p-2}F^a\phi^q. 
	\end{eqnarray*}
	Performing integration by parts again, we obtain
	\begin{eqnarray*}
		-(a+R_0+\frac{n}{2})\int_{M}\left\vert Rm \right\vert ^{p}F^{a}\phi ^{q} 
		&\leq& -2p\int_M R_{ijkh}(\na_h\na_lF) R_{ijkl} \ \! |Rm|^{p-2}F^a\phi^q \\
		&& -2p\int_M  (R_{ijkh}\na_hF) (\na_lR_{ijkl}) \ \!  |Rm|^{p-2}F^a\phi^q \\
		&& -2p\int_M (R_{ijkh}\na_hF) R_{ijkl} \na_l(|Rm|^{p-2})F^a\phi^q \\
		&& -2pa\int_M  |R_{ijkl} \na_lF|^2 |Rm|^{p-2}F^{a-1}\phi^q \\
		&& + \frac{4p^2}{r}\int_M  |R_{ijkl} \na_lF|^2 |Rm|^{p-2}F^{a}\phi^{q-1}.
	\end{eqnarray*}

	Since $Rc\geq 0$ implies $\na_i\na_jF \geq \frac{1}{2}g_{ij}$, it follows that 
	\begin{eqnarray*}
		-2p\int_M R_{ijkh}(\na_h\na_lF) R_{ijkl} |Rm|^{p-2}F^a\phi^q  \leq -p \int_{M}\left\vert Rm \right\vert ^{p}F^{a}\phi ^{q}.
	\end{eqnarray*}	
	Thus, by also using the last equality in Lemma 2.4, we get 
	\begin{eqnarray*}
		&&[p-(a+R_0+\frac{n}{2})]\int_{M}\left\vert Rm \right\vert ^{p}F^{a}\phi ^{q} \\
		&\leq & -2p\int_M (R_{ijkh} \na_hF) R_{ijkl} \na_l (|Rm|^{p-2})F^a\phi^q \\
		&&+	2p\int_{M} |R_{ijkl}\na_lF|^2 |Rm|^{p-2} F^{a}\phi^{q} \\
		&&+ \frac{4p^2}{r} \int_{M} |R_{ijkl}\na_lF|^2 |Rm|^{p-2} F^{a}\phi^{q-1} \\
        &=&I +II +III.
	\end{eqnarray*}
	
For the first term, by using Lemma 2.4 again, we have
	\begin{eqnarray*}
		I &=&-2p\int_M (R_{ijkh} \na_hF) R_{ijkl}(\na_l|Rm|^{p-2})F^a\phi^q \\
		&=& 2p\int_M (\na_jR_{ik}-\na_iR_{jk})R_{ijkl}(\na_l|Rm|^{p-2})F^a\phi^q\\
		&=& 2p(p-2)\int_M (\na_jR_{ik}-\na_iR_{jk}) R_{ijkl}(\na_l|Rm|)|Rm|^{p-3}F^a\phi^q \\
		&\leq&	4p^2\int_{M} |\na Rc| |\na Rm| |Rm|^{p-2} F^{a}\phi^{q} \\
		&\leq&  p \int_{M} |\na Rc|^2|Rm|^{p-1}F^{a+1}\phi^q + 4p^3\int_M |\na Rm|^2|Rm|^{p-3} F^{a-1}\phi^{q}.
	\end{eqnarray*}

	On the other hand, by Lemma \ref{diffequation}, we have
	\begin{eqnarray*} 
		II &=& 2p\int_M |R_{ijkl}\na_lF|^2 |Rm|^{p-2} F^a\phi^q \\
		& = & 2p\int_M |\na_i R_{jk}-\na_jR_{ik}|^2 |Rm|^{p-2} F^a\phi^q \\ 
		& \leq & 8p \int_M |\na Rc|^2 |Rm|^{p-2} F^a\phi^q \\
		& \leq & cp \int_M |\na Rc||\na Rm| |Rm|^{p-2} F^a\phi^q \\
		& \leq & p \int_M |\na Rc|^2 |Rm|^{p-1} F^{a+1}\phi^q + cp \int_M |\na Rm|^2 |Rm|^{p-3} F^{a-1}\phi^q.
	\end{eqnarray*}

	Finally, since $|\na F|^2\le F\le  r$ on $D(r)$, by Lemma \ref{diffequation} and Young's inequality,
	\begin{eqnarray*} 
		III &=& \frac{4p^2}{r} \int_{M} |R_{ijkl}\na_lF|^2 |Rm|^{p-2} F^{a}\phi^{q-1} \\
		& \leq & 4p^2 \int_{M} |R_{ijkl}\na_lF|^{\frac{2p}{p+1}} |R_{ijkl}\na_lF|^{\frac{2}{p+1}} |Rm|^{p-2} F^{a-1}\phi^{q-1} \\
		& \leq & 16p^2 \int_{M} |\na Rc|^{\frac{2p}{p+1}} |Rm|^{\frac{2}{p+1}}F^{\frac{1}{p+1}} |Rm|^{p-2} F^{a-1}\phi^{q-1} \\
		& = & 16p^2 \int_{M} |\na Rc|^{\frac{2p}{p+1}} |Rm|^{\frac{p(p-1)}{p+1}}F^{a-1+\frac{1}{p+1}}\phi^{q-1} \\
		& = & 16p^2 \int_{M} \left( |\na Rc|^{2} |Rm|^{p-1}F^{a+1}\phi^{q} \right)^{\frac{p}{p+1}}\cdot \left( F^{a-2p} \phi^{q-p-1} \right)^{\frac{1}{p+1}} \\
		& \leq & 2p \int_M |\na Rc|^2 |Rm|^{p-1} F^{a+1}\phi^q + c(p) \int_M F^{a-2p}\phi^{p-1} \\
		& \leq & 2p \int_M |\na Rc|^2 |Rm|^{p-1} F^{a+1}\phi^q + c(p).
	\end{eqnarray*}
	Here, in the last inequality, we have used the assumption $p>a+R_0+\frac{n}{2}+1$ and the fact that $(M, g)$ has at most Euclidean volume growth to deduce that $\int_M F^{a-2p} \leq c$.

	Therefore, 
	\begin{eqnarray*}
		&& [p-(a+R_0+\frac{n}{2})]\int_{M}\left\vert Rm \right\vert ^{p}F^{a}\phi ^{q} \\
		&\leq & 4p \int_{M} |\na Rc|^2|Rm|^{p-1}F^{a+1}\phi^q \\
		&& + cp^3\int_M |\na Rm|^2|Rm|^{p-3} F^{a-1}\phi^{q} +c(p).
	\end{eqnarray*}
	This completes the proof of Lemma \ref{lemma1}.

\end{proof}

\begin{remark}
	In the proof of Lemma \ref{lemma1}, as well as the proofs of Lemma \ref{lemma2} and Lemma \ref{lemma3} below, the constant $c(p)$ could be in the order of $p^p$ after applying Young's inequality.
\end{remark}

\smallskip
\begin{lemma} \label{lemma2}
	Let $(M^n, g, f)$ be an $n$-dimensional complete gradient expanding Ricci soliton with nonnegative Ricci curvature $Rc\geq 0$ and finite 
	{\it asymptotic scalar curvature ratio}
	$$\limsup_{r\to \infty}  R \ \!r^2< \infty. $$
	Suppose $p>a+\frac{n}{2}+1$, then
	\begin{eqnarray*}
		2 \int_M |\na Rc|^2 |Rm|^{p-1} F^{a+1}\phi^{q}
		&\leq & cp^3\int_M |\na Rm|^2|Rm|^{p-3}F^{a-1}\phi^{q} \\
		&&+ \frac{c}{p^2}\int_M |Rm|^pF^a\phi^q + c(p).
	\end{eqnarray*} 
\end{lemma}

\begin{proof}
	First of all, by Lemma 2.5 and direct computations, we obtain
	\begin{eqnarray*} 
		\Delta_f (|Rc|^2 |Rm|^{p-1}) &=& (\Delta_f |Rc|^2) |Rm|^{p-1} + |Rc|^2 \Delta_f (|Rm|^{p-1}) \\ 
		&&+2 \na (|Rc|^2) \cdot \na (|Rm|^{p-1})\\	
		&\geq & 2 |\na Rc|^2|Rm|^{p-1} -2p |Rc|^2 |Rm|^{p-1} -cp |Rc|^2 |Rm|^{p} \\
		&&-4p|\na Rc| |\na Rm| |Rc| |Rm|^{p-2}.
	\end{eqnarray*} 
	Consequently, 
	\begin{eqnarray*}
		&&2 \int_M |\na Rc|^2 |Rm|^{p-1} F^{a+1}\phi^{q} \\
		&\leq & \int_M \Delta_f(|Rc|^2|Rm|^{p-1})F^{a+1}\phi^{q} \\
		&&+ 2p\int_M  |Rc|^2|Rm|^{p-1}F^{a+1}\phi^{q} \\
		&&+ cp\int_M |Rc|^2|Rm|^{p}F^{a+1}\phi^{q} \\
		&&+ 4p\int_M |\na Rc||\na Rm||Rc||Rm|^{p-2}F^{a+1}\phi^{q} \\
		&=& I+II+III+IV.
	\end{eqnarray*}

	On one hand, using the quadratic decay of $Rc$ and Young's inequality, we have
	\begin{eqnarray*} 
		II &=&  2p \int_M |Rc|^2 |Rm|^{p-1} F^{a+1}\phi^{q} \\
		&\leq& cp\int_M |Rm|^{p-1} F^{a-1}\phi^{q} \\
		&\leq& \frac{1}{p^2} \int_M |Rm|^{p} F^{a}\phi^{q} + c(p)\int_M F^{a-p}\phi^{q} \\
		&\leq& \frac{1}{p^2}\int_M |Rm|^{p} F^{a}\phi^{q} + c(p),
	\end{eqnarray*}
	where, in the last inequality, we have $\int_M F^{a-p}<c$ due to (2.6) and $p>a+\frac{n}{2}+1$.
	
	Moreover, by the quadratic decay of $Rc$ and (\ref{decayofRic}), we get
	\begin{eqnarray*} 
		III &=&  cp \int_M |Rc|^2 |Rm|^{p} F^{a+1}\phi^{q} \\
		&\leq& \frac{c}{p^2} \int_M |Rm|^{p} F^{a}\phi^{q} + C.
	\end{eqnarray*}
	
	On the other hand, since $\Delta_f u=\Delta u-\na f\cdot \na u=\Delta u+\na F\cdot \na u$, by integration by parts, we have
	\begin{eqnarray*} 
		I &=& \int_M \Delta_f(|Rc|^2|Rm|^{p-1})F^{a+1}\phi^{q} \\
		&=&  \int_M  \Delta (|Rc|^2 |Rm|^{p-1})F^{a+1}\phi^{q} \\
		&&+ \int_M \na F \cdot \na(|Rc|^2 |Rm|^{p-1})F^{a+1}\phi^{q} \\
		&=&  \int_M |Rc|^2 |Rm|^{p-1} \Delta (F^{a+1}\phi^{q}) \\
		&&+ \frac {q}{r}  \int_M |Rc|^2 |Rm|^{p-1}  |\na F|^2 F^{a+1}\phi^{q-1} \\
		&&- \int_M |Rc|^2 |Rm|^{p-1} [\Delta F + (a+1)F^{-1} |\na F|^2]F^{a+1}\phi^{q}\\
		&\leq & \int_M |Rc|^2 |Rm|^{p-1} \Delta (F^{a+1}\phi^{q}) \\
		&&+ 2p \int_M |Rc|^2 |Rm|^{p-1} F^{a+1}\phi^{q-1} \\
		& = & I_A + I_B.
	\end{eqnarray*}
	Here, we have used the facts that $|\na F|^2\le F\le  r$ on $D(r)$ and $\Delta F=R+\frac{n}{2}\geq 0$.

	Now, by direct computations, we have  
	\begin{eqnarray*} 
		&&\Delta (F^{a+1}\phi^{q})\\
		&=& \Delta (F^{a+1})\phi^{q}  + F^{a+1}\Delta (\phi^{q})  +2\na F^{a+1} \cdot \na \phi^{q}\\
		&\leq & [(a+1)F^{a}\Delta F+a(a+1)F^{a-1}|\na F|^2]\phi^{q}  \\
		& & +F^{a+1} [q \phi^{q-1} \Delta\phi +q(q-1)\phi^{q-2} |\na \phi|^2] \\
		&\leq & cp^2 F^{a}\phi^{q} + 4p^2  F^{a}\phi^{q-2}\\
		&\leq & cp^2 F^{a}\phi^{q-2},
	\end{eqnarray*}
	where we have used the facts that $\na F\cdot \na\phi\leq 0, \ \Delta \phi\leq 0, \ \Delta F\leq R_0+n/2$, $|\na F|^2 \leq F$, and $F |\na \phi|^2\leq 1$. 
	
    Hence, by Young's inequality, the quadratic decay of $Rc$ and (\ref{decayofRic}), we obtain
	\begin{eqnarray*} 
		I_A &=&  \int_M |Rc|^2 |Rm|^{p-1} \Delta (F^{a+1}\phi^{q}) \\
		&\leq& cp^2\int_M |Rc|^2 |Rm|^{p-1} F^{a}\phi^{q-2} \\
		&\leq& \frac{c}{p^2}\int_M |Rm|^{p-1} F^{a-1}\phi^{q-2} + C \\
		&\leq& \frac{1}{p^2} \int_M |Rm|^{p} F^{a}\phi^{q} + c(p)\int_M F^{a-p}\phi^{q-2p} + C \\
		&\leq& \frac{1}{p^2} \int_M |Rm|^{p} F^{a}\phi^{q} + c(p).
	\end{eqnarray*}

	Similarly, 
	\begin{eqnarray*} 
		I_B &=& 2p \int_M |Rc|^2 |Rm|^{p-1}  F^{a+1}\phi^{q-1} \\
		&\leq& cp\int_M |Rm|^{p-1} F^{a-1}\phi^{q-1} \\
		&\leq& \frac{1}{p^2}\int_M |Rm|^{p} F^{a}\phi^q + c(p)\int_M F^{a-p}\phi^{q-p} \\
		&\leq& \frac{1}{p^2}\int_M |Rm|^{p} F^{a}\phi^{q} + c(p).
	\end{eqnarray*}
	
	Finally, 
	\begin{eqnarray*} 
		IV &=&  4p\int_M |\na Rc||\na Rm||Rc||Rm|^{p-2}F^{a+1}\phi^{q} \\
		&\leq& 4p^3\int_M |\na Rm|^2|Rm|^{p-3} F^{a-1}\phi^{q} \\
		&&+ \frac{1}{p} \int_M |\na Rc|^2|Rc|^2|Rm|^{p-1}F^{a+3}\phi^{q} \\
		&\leq& 4p^3\int_M |\na Rm|^2|Rm|^{p-3} F^{a-1}\phi^{q} \\
		&&+ \frac{c}{p}\int_M |\na Rc|^2|Rm|^{p-1}F^{a+1}\phi^{q}.
	\end{eqnarray*}
	By combining the above estimates, we have completed the proof of Lemma \ref{lemma2}.

\end{proof}

\begin{lemma} \label{lemma3}
	Let $(M^n, g, f)$ be an $n$-dimensional complete gradient expanding Ricci soliton with nonnegative Ricci curvature $Rc\geq 0$.
	Suppose $p>a+\frac{n}{2}+1$, then
	\begin{eqnarray*}
		2\int_M |\na Rm|^2 |Rm|^{p-3}F^{a-1}\phi^{q}
		&\leq & \frac{c}{p^5}\int_M |Rm|^{p} F^{a}\phi^{q} + c(p).
	\end{eqnarray*}
\end{lemma}

\begin{remark}
	Note that, like Lemma 3.1,  Lemma \ref{lemma3} does not require the {\it finite asymptotic scalar curvature ratio} assumption either.
\end{remark}

\begin{proof}
First of all, note that
\begin{eqnarray*}
	2|\na Rm|^2 \leq \Delta|Rm|^2 + \na F\cdot \na |Rm|^2 + 2|Rm|^2 + c|Rm|^3.
\end{eqnarray*}

Therefore, by integration by parts, we have
\begin{eqnarray*}
	&& 2\int_M |\na Rm|^2 |Rm|^{p-3}F^{a-1}\phi^{q} \\
	& \leq & \int_M (\Delta|Rm|^2)|Rm|^{p-3}F^{a-1}\phi^{q} \\
	&& + \int_M (\na F\cdot \na |Rm|^2) |Rm|^{p-3}F^{a-1}\phi^{q} \\
	&& + 2\int_M |Rm|^{p-1}F^{a-1}\phi^{q}  \\
	&& + c\int_M |Rm|^{p}F^{a-1}\phi^{q} \\
	& \leq & -(a-1)\int_M (\na F\cdot \na |Rm|^2) |Rm|^{p-3}F^{a-2}\phi^{q} \\
	&& + \frac{q}{r} \int_M (\na F\cdot \na |Rm|^2) |Rm|^{p-3}F^{a-1}\phi^{q-1} \\
	&& + \int_M (\na F\cdot \na |Rm|^2) |Rm|^{p-3}F^{a-1}\phi^{q} \\
	&& + 2\int_M |Rm|^{p-1}F^{a-1}\phi^{q}  \\
	&& + c\int_M |Rm|^{p}F^{a-1}\phi^{q} \\
	&=& I+II+III+IV+V.
\end{eqnarray*}

It follows from integration by parts, $\Delta F\leq R_0+n/2$, $|\na F|^2 \leq F$ and (\ref{decayofRic}) that
\begin{eqnarray*}
	I &=&  -(a-1)\int_M ( \na F\cdot \na |Rm|^2 ) |Rm|^{p-3}F^{a-2}\phi^{q} \\
	&=&  -\frac{2(a-1)}{p-1}\int_M ( \na F\cdot \na |Rm|^{p-1} ) F^{a-2}\phi^{q} \\
	&=&  \frac{2(a-1)}{p-1}\int_M |Rm|^{p-1}[\Delta F+(a-2)F^{-1}|\na F|^2] F^{a-2}\phi^{q} \\
	&&- \frac{2(a-1)q}{(p-1)r}\int_M |Rm|^{p-1} |\na F|^2  F^{a-2}\phi^{q-1} \\
	&\leq& c\int_M |Rm|^{p-1} F^{a-1}\phi^{q} +C \\
	&\leq& \frac{1}{p^5} \int_M |Rm|^{p} F^{a}\phi^{q} + c(p)\int_M F^{a-p}\phi^{q} +C \\
	&\leq& \frac{1}{p^5} \int_M |Rm|^{p} F^{a}\phi^{q} + c(p),
\end{eqnarray*}
where, in the last two inequalities, we have again used Young's inequality, (2.6), and $p>a+\frac{n}{2}+1$.

Similarly, for $r \geq 1$, as $\Delta F=R+\frac{n}{2}> 0$ and $|\na F|^2 \leq F \le r$ \ \! over $D(r)$, 
\begin{eqnarray*}
	II + III &=&  \frac{q}{r}\int_M (\na F\cdot \na |Rm|^2) |Rm|^{p-3}F^{a-1}\phi^{q-1} \\
	&& + \int_M (\na F\cdot \na |Rm|^2) |Rm|^{p-3}F^{a-1}\phi^{q}\\
	&\leq& (2p+1)\int_M (\na F\cdot \na |Rm|^2) |Rm|^{p-3}F^{a-1}\phi^{q-1} \\
	&=&  \frac{2(2p+1)}{p-1}\int_M ( \na F\cdot \na |Rm|^{p-1} ) F^{a-1}\phi^{q-1} \\
	&=&  -\frac{2(2p+1)}{(p-1)}\int_M |Rm|^{p-1}[\Delta F+(a-1)F^{-1}|\na F|^2] F^{a-1}\phi^{q-1} \\
	&&+ \frac{2(2p+1)(q-1)}{(p-1)r}\int_M |Rm|^{p-1} |\na F|^2  F^{a-1}\phi^{q-2} \\
	&\leq& \frac{2(2p+1)(q-1)}{(p-1)}\int_M |Rm|^{p-1} F^{a-1}\phi^{q-2} \\
	&\leq& \frac{1}{p^5} \int_M |Rm|^{p} F^{a}\phi^{q} + c(p)\int_M F^{a-p}\phi^{q-2p} \\
	&\leq& \frac{1}{p^5} \int_M |Rm|^{p} F^{a}\phi^{q} + c(p).
\end{eqnarray*}

On the other hand, by Young's inequality, (2.6), and $p>a+\frac{n}{2}+1$, we get
\begin{eqnarray*}
	IV &=&  2\int_M |Rm|^{p-1} F^{a-1}\phi^{q} \\
	&\leq& \frac{1}{p^5} \int_M |Rm|^{p} F^{a}\phi^{q} + c(p),
\end{eqnarray*}

and 
\begin{eqnarray*} 
	V &=&  c \int_M |Rm|^{p} F^{a-1}\phi^{q} \\
	&\leq& \frac{c}{p^5}\int_M |Rm|^{p} F^{a}\phi^{q} +C \\ 
	&\leq& \frac{1}{p^5}\int_M |Rm|^{p} F^{a}\phi^{q} +c(p),
\end{eqnarray*}
where we have used (\ref{decayofRic}) in deriving the first inequality for $V$.

Combining all the estimates above, the proof of Lemma \ref{lemma3} is completed.

\end{proof}

Now we can conclude the proof of {\bf Proposition 3.1.} 

\proof
For any constant $a>0$,  let $p>a+R_0+\frac{n}{2}+c$ for some constant $c\geq 1$. Then, by combining Lemma \ref{lemma1}, Lemma \ref{lemma2} and Lemma \ref{lemma3}, Proposition \ref{prop} follows immediately.

\hfill $\Box$

\section{The proof of Main theorem}

In this section, we use the integral estimate in Section 3 and the De Giorgi-Nash-Moser iteration to prove our main result on the pointwise decay estimate of the curvature tensor $Rm$ as stated in the introduction (see also Theorem \ref{maintheorem}). 

\begin{theorem} \label{mainthm}
	Let $(M^n, g, f)$ be an $n$-dimensional complete gradient expanding Ricci soliton with nonnegative Ricci curvature $Rc\geq 0$ and finite 
	{\it asymptotic scalar curvature ratio}
	\begin{equation*} \label{aympscalar}
		\limsup_{r\to \infty}  R \ \!r^2< \infty,
	\end{equation*}
	Then $(M^n, g, f)$ has finite $\alpha$-asymptotic curvature ratio for any $0<\alpha<2$,
	\begin{equation}
		A_{\alpha} := \limsup_{r\to \infty} |Rm| \ \! r^{\alpha}< \infty.
	\end{equation}
Furthermore, there exist constant $C>0$ depending on $n$ and the geometry of $(M^n, g, f)$, sequences $\{r_j\} \to \infty$ and $\{\alpha_j\} \to 2$ such that 
\begin{equation*} 
	|Rm|(x) \leq  C (r(x)+1)^{-\alpha_j}
\end{equation*}
for any $x \in M\setminus B(x_0, r_j+1)$.
\end{theorem}

\begin{proof}  As in Munteanu-Wang \cite{MW17}, we now combine Proposition 3.1 and the De Giorgi-Nash-Moser  iteration to obtain the pointwise curvature tensor decay estimate.  

First of all, for any $p>0$ large and $a>0$ such that $p>a+R_0+\frac{n}{2}+c>a+\frac{n}{2}+1$, by Proposition \ref{prop} we have
\begin{equation*}
	\int_{M}\left\vert Rm \right\vert ^{p}F^{a}\phi ^{q} \leq c(p).
\end{equation*}
Since the cut-off function $\phi \geq \frac{1}{2}$ on $D(r/2)$ by (\ref{cut-off}), it follows that
\begin{equation*}
	\int_{D(r/2)}\left\vert Rm \right\vert ^{p}F^{a} \leq c(p)
\end{equation*}
for $r>r_0$ arbitrarily large. 
Hence,
\begin{equation*}
	\int_{M}\left\vert Rm \right\vert ^{p}F^{a} \leq c(p).
\end{equation*}
Note that if we define 
\begin{equation*}
	I(r):=\int_{D(r)}|Rm|^pF^a, 
\end{equation*}
then clearly $I(r)$ is increasing in $r$ and 
\begin{equation*}
	\lim_{r\rightarrow \infty}I(r) = \int_{M}\left\vert Rm \right\vert ^{p}F^{a} \leq c(p).
\end{equation*}
Thus, for any fixed $p>0$ large there exists a constant $r_p>r_0$ such that
\begin{equation*}
	\int_{M\backslash D(r_p)} |Rm|^pF^a \leq 1.
\end{equation*}
Therefore, by Lemma \ref{potencialfunction}, we have
\begin{equation} \label{boundfor|Rm|}
	\int_{B_x(1)} |Rm|^p \leq c \left(r(x)+1\right)^{-2a}
\end{equation}
for any $x\in M\backslash D(r_p+1)$.

Next, we apply the Moser iteration to get the pointwise decay estimate for $Rm$ from (\ref{boundfor|Rm|}). We start by deriving an inequality satisfied by $\Delta |Rm|^2$. 

From Lemma \ref{diffeqofnorm}, we note that
$$ \Delta_f |Rm|^2 \geq 2|\na Rm|^2 - 2|Rm|^2 - c|Rm|^3. $$
Also, by using the Cauchy-Schwarz inequality and Kato's inequality, we have
\begin{eqnarray*} 
	\na F \cdot \na |Rm|^2 &=&  2|Rm|\na F\cdot \na |Rm| \\
	&\leq& \frac{1}{2} |Rm|^2 |\na F|^2 + 2|\na Rm|^2.
\end{eqnarray*}
Thus, 
\begin{eqnarray*} 
	\Delta |Rm|^2 &\geq&  2|\na Rm|^2 - 2|Rm|^2 - c|Rm|^3 - \na F \cdot \na |Rm|^2 \\
	&\geq& - 2|Rm|^2 - c|Rm|^3 - \frac{1}{2}(F-R-\frac{n}{2}) |Rm|^2 \\
	&\geq& -c(F+|Rm|)|Rm|^2 \\
	&=& -u|Rm|^2,
\end{eqnarray*}
where $u:=c(F+|Rm|)$.

\smallskip
By  Lemma \ref{avr} and Lemma  \ref{sobolev},  or by the Sobolev inequality in \cite{Saloff} together with the non-collapsing estimate of Carrillo and Ni in Lemma \ref{non-collapsing}, we know that the Sobolev inequality holds on the unit geodesic ball $B_x(1)$, with 
the Sobolev constant 
    $C_s$
independent of $x\in M$. Therefore, by applying the Moser iteration (see \cite{PLi} or \cite{PLi93}), we have
\begin{equation} \label{Nash-Moser_iteration2}
	|Rm|(x) \leq C_0\left( \int_{B_x(1)} u^n+1 \right)^{\frac{1}{p}}\left( \int_{B_x(1)} |Rm|^p \right)^{\frac{1}{p}},
\end{equation}
where $C_0>0$ depends only on $n$ and $C_s$. 
Note that, by (\ref{boundfor|Rm|}) and the Bishop volume comparison, we have
\begin{eqnarray*} 
		\int_{B_x(1)} |Rm|^n & \leq & \left( \int_{B_x(1)} |Rm|^p \right)^{\frac{n}{p}}\vol(B_x(1))^{\frac{p-n}{p}} \\
                 & \leq & c\left( r(x)+1 \right)^{-2a\cdot \frac{n}{p}}
\end{eqnarray*}
for any $x\in M\backslash D(r_p+1)$. 
Hence,
\begin{eqnarray} \label{boundforu^n2}
	\int_{B_x(1)} u^n
	&=& c\int_{B_x(1)} (F+|Rm|)^n  \\
	&\leq& c\int_{B_x(1)} F^n + c\int_{B_x(1)} |Rm|^n \notag \\
	&\leq& c(r(x)+1)^{2n}. \notag
\end{eqnarray}

Now, for $p>0$ large, we take
\begin{equation} \label{arelativetop}
	a=p-(\frac{n}{2}+R_0+c+1).
\end{equation}
Then, by (\ref{boundfor|Rm|})-(\ref{boundforu^n2}), we have
\begin{eqnarray} \label{pointwisefor|Rm|}
	|Rm|(x) &\leq&  C_0 \left( \int_{B_x(1)} u^n+1 \right)^\frac{1}{p} \left( \int_{B_x(1)} |Rm|^{p}\right)^\frac{1}{p} \notag \\
	&\leq& C_0 (r(x)+1)^{-\frac{2(a-n)}{p}}  
\end{eqnarray}
for $x \in M\backslash D(r_p+1)$. 

On the other hand, for any $\alpha \in (0,2)$,  then for $p$ sufficiently large we have 
$$ \frac{a-n}{p} = 1-\frac{\frac{n}{2}+R_0+c+1+n}{p} \geq \frac{\alpha}{2}. $$
Now, for $\alpha ,\ p$ and $a$ as above, by (\ref{pointwisefor|Rm|}) we have
\begin{eqnarray*} 
	|Rm|(x) &\leq&  C_0 (r(x)+1)^{-\frac{2(a-n)}{p}} \\
		&=& C_0 (r(x)+1)^{-\alpha}
\end{eqnarray*}
for any $x \in M\backslash D(r_p+1)$.

Furthermore, note also that we have $r_p \rightarrow \infty$ as $p \rightarrow \infty$. Thus, if we take $p=j \in {\mathbb N}$ and set 
\[ \alpha_j= \frac{2(a-n)}{p} = 2-\frac{3n+2R_0+2c+2}{j} \to 2,\]
then there exists a sequence $\{r_j\} \to \infty$  such that
\begin{equation*} 
	|Rm|(x) \leq  C_0 (r(x)+1)^{-\alpha_j}
\end{equation*}
for any $x \in M\backslash D(r_j+1)$. 

This completes the proof of Theorem \ref{mainthm}.

\end{proof}

\smallskip
In fact, as we mentioned in Remark 1.1, the same proof can be used to prove the following more general  curvature decay estimate.

\begin{theorem} \label{subalphadecaythm}
   	Let $(M^n, g, f)$ be an $n$-dimensional complete gradient expanding Ricci soliton with nonnegative Ricci curvature $Rc\geq 0$ and finite $\alpha_0$-asymptotic curvature ratio for any $0<\alpha_0\leq 2$,
	\begin{equation*} 
		\limsup_{r\to \infty}  R \ \!r^{\alpha_0}< \infty. 
	\end{equation*}
	Then, $(M^n, g, f)$ has finite $\alpha$-asymptotic curvature ratio for any $0<\alpha<\alpha_0$,  
	\begin{equation*}
		A_{\alpha} := \limsup_{r\to \infty} |Rm| \ \! r^{\alpha}< \infty.
	\end{equation*}
Furthermore, there exist constant $C>0$ depending on $n$ and the geometry of $(M^n, g, f)$, sequences $\{r_j\} \to \infty$ and $\{\alpha_j\} \to \alpha_0$ such that 
\begin{equation*} 
	|Rm|(x) \leq  C (r(x)+1)^{-\alpha_j}
\end{equation*}
for any $x \in M\setminus B(x_0, r_j+1)$.
\end{theorem}

\begin{proof} 
	For any $\alpha_0 \in (0,2]$, let $\epsilon := \frac{\alpha_0}{2}$, then $\epsilon \in (0,1]$. For $p > a+R_0+\frac{n}{2}+1$, by following the same argument as in Lemma \ref{lemma1} and using Lemma \ref{diffequation},  we have
	\begin{eqnarray*}
		[1-p^{-1}(a+R_0+\frac{n}{2})]\int_{M}\left\vert Rm \right\vert ^{p}F^{a}\phi ^{q} 
		&\leq & 4 \int_M |\na Rc|^2 |Rm|^{p-1} F^{a+\epsilon}\phi^q  \\
&& + \ cp^2 \int_{M} |\na Rm|^2 |Rm|^{p-3} F^{a-\epsilon} \phi ^{q} \\ 
&& + \ c(p).
	\end{eqnarray*}
	Also, note that the same argument as in the proofs of Lemma \ref{lemma2} and Lemma \ref{lemma3} give us the following: if $\epsilon p > a+\frac{n}{2}+1$, then we have
	\begin{eqnarray*}
		2 \int_M |\na Rc|^2 |Rm|^{p-1} F^{a+\epsilon}\phi^{q}
		&\leq & cp^3\int_M |\na Rm|^2|Rm|^{p-3}F^{a-\epsilon}\phi^{q} \\
		&&+ \frac{c}{p^2}\int_M |Rm|^pF^a\phi^q + c(p),
	\end{eqnarray*} 
	and
	\begin{eqnarray*}
		2\int_M |\na Rm|^2 |Rm|^{p-3}F^{a-\epsilon}\phi^{q}
		&\leq & \frac{c}{p^5}\int_M |Rm|^{p} F^{a}\phi^{q} + c(p).
	\end{eqnarray*}
	By combining the estimates above, we see that if $\epsilon p > a+\frac{n}{2}+1$ and $p>a+\frac{n}{2}+R_0+c$ then we have
	\begin{equation*}
		[1-p^{-1}(a+\frac{n}{2}+R_0+c)]\int_{M}\left\vert Rm \right\vert ^{p}F^{a}\phi ^{q} \leq c(p).
	\end{equation*}

	As in the proof of Theorem \ref{mainthm}, for any fixed $p>0$ large, there exists a constant $r_p>r_0$ such that
	\begin{equation*}
		\int_{M\backslash D(r_p)} |Rm|^pF^a \leq 1.
	\end{equation*}
	Therefore, by Lemma \ref{potencialfunction}, for any $x\in M\backslash D(r_p+1)$, we get
	\begin{equation*} 
		\int_{B_x(1)} |Rm|^p \leq c\left(r(x)+1\right)^{-2a}.
	\end{equation*}

	For any $p>0$ large, we take
	\begin{equation} \label{arelativetop3}
		a=\epsilon p-(\frac{n}{2}+R_0+c+1).
	\end{equation}
	Then by following the same proof as in Theorem \ref{mainthm}, we have
	\begin{eqnarray} \label{pointwisefor|Rm|3}
		|Rm|(x) &\leq&  C_0 \left( \int_{B_x(1)} u^n+1 \right)^\frac{1}{p} \left( \int_{B_x(1)} |Rm|^{p}\right)^\frac{1}{p} \notag \\
		&\leq& C_0 (r(x)+1)^{-\frac{2(a-n)}{p}}  
	\end{eqnarray}
	for $x \in M\backslash D(r_p+1)$. 
	
	We note that for any $\alpha \in (0,\alpha_0)$, when $p$ is sufficiently large, we have
	$$ \frac{a-n}{p} = \epsilon-\frac{\frac{n}{2}+R_0+c+1+n}{p}
	\geq \frac{\alpha}{2}. $$
	Now, for $\alpha ,\ p$, $a$ as above and any $x \in M\backslash D(r_p+1)$, 
by (\ref{pointwisefor|Rm|3}) we obtain
	\begin{eqnarray*} 
		|Rm|(x) &\leq&  C_0 (r(x)+1)^{-\frac{2(a-n)}{p}} \\
		&=& C_0 (r(x)+1)^{-\alpha}.
	\end{eqnarray*}

	Moreover, as in the proof of Theorem \ref{mainthm}, if we take $p=j \in {\mathbb N}$ and set 
	\[ \alpha_j= \frac{2(a-n)}{p} = \alpha_0-\frac{3n+2R_0+2c+2}{j} \to \alpha_0,\]
	then there exists a sequence $\{r_j\} \to \infty$  such that
	\begin{equation*} 
		|Rm|(x) \leq  C_0 (r(x)+1)^{-\alpha_j}
	\end{equation*}
	for any $x \in M\backslash D(r_j+1)$. 
	
	This completes the proof of Theorem \ref{subalphadecaythm}.
	
\end{proof}

\end{document}